\documentclass[12pt,twoside,reqno]{amsart}
\usepackage{pifont}
\usepackage{amsmath}
\usepackage{amsthm}
\usepackage{amsfonts}
\usepackage{amssymb}
\usepackage{latexsym}
\usepackage{amstext}
\usepackage{array}
\usepackage{graphicx}

\date{}
\allowdisplaybreaks[4] \footskip=15pt
\renewcommand{\uppercasenonmath}[1]{}

\newtheorem{thm}[subsection]{Theorem}
\newtheorem{cor}[subsection]{Corollary }
\newtheorem{Def}[subsection]{Definition}
\newtheorem{lem}[subsection]{Lemma}

\newtheorem{prop}[subsection]{Proposition}
\newtheorem{exm}[subsection]{Example}
  
\newcommand{\bthm}{\begin{thm} }
\newcommand{\ethm}{\end{thm} }
\newcommand{\bpro}{\begin{prop}}
\newcommand{\epro}{\end{prop}}
\newcommand{\bdf}{\begin{Def}}
\newcommand{\edf}{\end{Def}}
\newcommand{\bexm}{\begin{exm}}
\newcommand{\eexm}{\end{exm}}
\newcommand{\blem}{\begin{lem}}
\newcommand{\elem}{\end{lem}}
\newcommand{\bpf}{\begin{proof}}
\newcommand{\epf}{\end{proof}}
\newcommand{\bcor}{\begin{cor}}
\newcommand{\ecor}{\end{cor}}
\newcommand{\ba}{\begin{array}}
\newcommand{\ea}{\end{array}}
\newcommand{\bea}{\begin{eqnarray}}
\newcommand{\eea}{\end{eqnarray}}

\newcommand{\brem}{\begin{remark}}
\newcommand{\erem}{\end{remark}}

\newtheorem{theorem}{Theorem}[section]

\newtheorem{corollary}[theorem]{Corollary}
\newtheorem{lemma}[theorem]{Lemma}
\theoremstyle{definition}
\newtheorem{definition}[theorem]{Definition}
\newtheorem{example}[theorem]{Example}
\theoremstyle{remark}
\newtheorem{remark}[theorem]{Remark}

\newcommand{\Z}{\mathbb Z}
\newcommand{\F}{\mathbb F}
\newcommand{\Cay}{\operatorname{Cay}}
\newcommand{\diam}{\operatorname{diam}}
\newcommand{\popt}{\pi_{\mathrm{opt}}}
\newcommand{\supp}{\operatorname{supp}}
\newcommand{\ind}{\mathbf 1}
\newcommand{\rhothree}{\rho_3}
\newcommand{\wt}{\operatorname{wt}}

\newsavebox{\tablebox}

\begin{document}

\linespread{1}

\begin{center}
{\large  \bf Pebbling Numbers and Generator Reinforcement of Cayley Graphs}\\

 \vskip 0.8cm
{\small
 Yuan Yang \footnote{School of Mathematics, Nanjing University, Nanjing 210093, P.R.China\\
  \indent E-mail addresses: zuohuiwoziji1@126.com(Y. Yang)} \ \
 Bocheng Li \footnote{Corresponding author. Army Engineering University of PLA, Nanjing 210042, P.R.China\\
 \indent E-mail addresses: 1358736618@qq.com(B. Li)}\ \
 Qingzhong Ji \footnote{School of Mathematics, Nanjing University, Nanjing 210093, P.R.China \\
  \indent E-mail addresses: qingzhji@nju.edu.cn(Q. Ji)}}
\vskip 3mm
\end{center}

{\bf Abstract:} {\small We study inverse-closed generator additions of minimum size that make the optimal pebbling number of a finite Cayley graph at most three while preserving its specified regular action. A weighted vertex-cover formulation on inverse orbits yields an $O(|\Gamma|^2)$ algorithm for abelian groups in the explicit multiplication-table model. We derive formulas for central involutions and cubelike graphs, with an inverse-closure correction determined by square roots in the uncovered set. The exact thresholds are $2^{d-1}-d$ for hypercubes ($d\ge1$), and $3$ for standard Keller graphs when $d=2$ and $2^{d-1}$ when $d\ge3$. Every admissible augmentation of either family has ordinary pebbling number equal to its order.}

\section{\bf Introduction and conventions}\label{sec:intro}

A pebble distribution on a graph $G$ is a function
$D:V(G)\to\Z_{\ge0}$, with size
$|D|=\sum_{v\in V(G)}D(v)$.
A legal pebbling move removes two pebbles from a vertex containing
at least two pebbles and adds one pebble to an adjacent vertex.
A vertex $r$ is \emph{reachable} from $D$ if a finite sequence of
legal moves leaves at least one pebble at $r$.
A distribution is \emph{solvable} if every vertex is reachable,
separately, from the original distribution.
Different targets may use different move sequences; simultaneous
coverage is not required.
The ordinary pebbling number $\pi(G)$ is the least integer $t$ such
that every distribution of size $t$ is solvable.
The optimal pebbling number $\popt(G)$ is the minimum size of a
solvable distribution.

All graphs are finite, simple, and connected unless stated otherwise,
and $\pi(K_1)=\popt(K_1)=1$.
We write $\ind_{\{P\}}$ for the indicator of a condition $P$:
it equals one when $P$ holds and zero otherwise.
For a vertex $v$, write $N_G(v)$ and $N_G[v]$ for its open and
closed neighborhoods, respectively, and suppress the subscript when
the graph is fixed.
An edge $uv$ is \emph{dominating} if
$N[u]\cup N[v]=V(G)$.
For a connected graph $G$ with at least two vertices, the following
elementary small-distribution criteria hold:
\begin{equation}\label{eq:small}
\begin{split}
 \popt(G)=2
 &\quad\Longleftrightarrow\quad
 G\text{ has a universal vertex},\\
 \popt(G)\le3
 &\quad\Longleftrightarrow\quad
 G\text{ has a dominating edge}.
\end{split}
\end{equation}
Accordingly, $\popt(G)=3$ holds precisely when $G$ has a dominating
edge but no universal vertex.
For a connected noncomplete regular graph, the absence of a
universal vertex is automatic. The second equivalence in
\eqref{eq:small} holds without regularity; the noncomplete regular
hypothesis only rules out a universal vertex.
For a graph of diameter two, placing four pebbles at one vertex
gives a solvable distribution. Thus these criteria yield the
characterization in \cite{Muntz2007}: the optimal pebbling number
is four precisely when there is no dominating edge; see also
\cite[Section~2]{Czygrinow2019}.

Write $\kappa(G)$ for the vertex-connectivity of $G$.
The Clarke--Hochberg--Hurlbert theorem states that every
$3$-connected graph of diameter two is Class~$0$, meaning that
$\pi(G)=|V(G)|$~\cite{CHH1997}.
The Watkins--Mader bound gives
$\kappa(G)\ge 2(k+1)/3$ for every connected noncomplete
vertex-transitive graph of degree $k$
\cite{Watkins1970,Mader1970}.
Thus, if $k\ge3$, such a graph is $3$-connected.
The noncomplete connected vertex-transitive graphs of degree two
and diameter at most two are $C_4$ and $C_5$.
Together with $\pi(K_n)=n$ for $n\ge1$ and the classical values
$\pi(C_4)=4$ and $\pi(C_5)=5$, this yields
$$
 G\text{ vertex-transitive},\quad \diam(G)\le2
 \quad\Longrightarrow\quad \pi(G)=|V(G)|.
$$
The ordinary pebbling formulas for cycles are due to \cite{PSV1995}:
$$
 \pi(C_{2m})=2^m \quad(m\ge2),\qquad
 \pi(C_{2m+1})
 =2\left\lfloor\frac{2^{m+1}}3\right\rfloor+1
 \quad(m\ge1),
$$
and the optimal pebbling formula is proved in \cite{Bunde2008}:
$$
 \popt(C_n)=\left\lceil\frac{2n}{3}\right\rceil
 \qquad(n\ge3).
$$

Let $\Gamma$ be a finite group of order $N\ge2$, with identity $e$,
and let
$$
 S=S^{-1}\subseteq\Gamma\setminus\{e\},
 \qquad \langle S\rangle=\Gamma.
$$
We use the right Cayley graph
$$
 G=\Cay(\Gamma,S),\qquad
 g\sim h\quad\Longleftrightarrow\quad g^{-1}h\in S.
$$
Put $k=|S|$ and $T=\{e\}\cup S$.
Thus $N[g]=gT$, and left multiplication acts regularly by
automorphisms.
Moreover, $\diam(G)\le2$ if and only if $T^2=\Gamma$.
In this case, the pebbling numbers are completely determined by
\begin{equation}\label{eq:diameter-cayley}
 \pi(G)=N, \ \
 \popt(G)=
 \begin{cases}
  2,
   & T=\Gamma,\\
  3,
   & T\ne\Gamma
     \text{ and }T\cup sT=\Gamma
     \text{ for some }s\in S,\\
  4,
   & \text{otherwise}.
 \end{cases}
\end{equation}

Indeed, the closed-neighborhood identity needed for this
specialization is
$$
 N[g]\cup N[gs]=g(T\cup sT).
$$
For $s\in S$, put $c(s)=|S\cap sS|$.
Since $T\cap sT$ is the disjoint union of $\{e,s\}$ and
$S\cap sS$, we have
$$
 |T\cup sT|=2k-c(s).
$$
Consequently, in the noncomplete case, the minimum number of
vertices outside the closed-neighborhood union of an edge is
$$
 \min_{s\in S}\bigl(N-|T\cup sT|\bigr)
 =N-2k+\min_{s\in S}c(s).
$$
This quantity is zero precisely when $\popt(G)=3$.
In particular, $N>2k$ excludes a dominating edge.

For a proper nontrivial subgroup $H<\Gamma$, let
$G_H=\Cay(\Gamma,\Gamma\setminus H)$.
This graph is complete multipartite, with the left cosets of $H$
as its parts.
By~\eqref{eq:diameter-cayley},
$$
 \bigl(\pi(G_H),\popt(G_H)\bigr)=(N,3).
$$
In particular, the Cayley graph generated by all reflections in
the dihedral group of order $2m$ is $K_{m,m}$, and
$$
 \bigl(\pi(K_{m,m}),\popt(K_{m,m})\bigr)=(2m,3)
 \qquad(m\ge2).
$$

For $m,n\ge2$, the Cartesian product $K_m\square K_n$ is a
Cayley graph of diameter two, with
$$
 \pi(K_m\square K_n)=mn,\qquad
 \popt(K_m\square K_n)=
 \begin{cases}
  3,&\min(m,n)=2,\\
  4,&m,n\ge3.
 \end{cases}
$$
The second case is also contained in
\cite[Corollary~3.10]{HHH2011}; the first follows immediately
from~\eqref{eq:small}.
For an edge whose endpoints differ only in the first coordinate,
the union of their closed neighborhoods misses $(m-2)(n-1)$ vertices.
If the endpoints differ only in the second coordinate, it misses
$(n-2)(m-1)$ vertices.

For example, the same group $\Z_8$ supports Cayley graphs with
the following different pebbling numbers.

\begin{table}[htbp]
  \centering
  \begin{tabular}{@{}lccc@{}}
    \hline
    Generating set & Graph & $\pi$ & $\popt$\\
    \hline
    $\{1,7\}$ & $C_8$ & $16$ & $6$\\
    $\{1,3,5,7\}$ & $K_{4,4}$ & $8$ & $3$\\
    $\{1,4,7\}$ & $\Cay(\Z_8,\{1,4,7\})$ & $8$ & $4$\\
    $\Z_8\setminus\{0\}$ & $K_8$ & $8$ & $2$\\
    \hline
  \end{tabular}
\caption{Pebbling and optimal pebbling numbers for Cayley graphs on $\Z_8$}
  \label{tab:pebbling-Z8}
\end{table}

In the third row, $T+T=\Z_8$ and $8>2|S|$.
The other rows follow from the cited cycle formulas and the
preceding specializations.

For a coordinate vector $x$ in $\F_2^d$ or $\Z_4^d$, let
$$
 \supp(x)=\{i:x_i\ne0\},\qquad
 \wt(x)=|\supp(x)|,
$$
and write $e_i$ for the $i$th coordinate vector.
For $d\ge1$, let
$$
 Q_d=\Cay(\F_2^d,\{e_1,\ldots,e_d\}).
$$
Chung~\cite{Chung1989} proved that $\pi(Q_d)=2^d$.
The hypercube lower bound of Moews, as presented in
\cite{Bunde2008}, together with the optimal
Cartesian-product inequality
\cite[Corollary~2.7]{HHH2011}, gives
$$
 \left\lceil(4/3)^d\right\rceil
 \le \popt(Q_d)
 \le 3^{\lfloor d/2\rfloor}2^{d\bmod2}.
$$
Pachter's 2026 preprint \cite{Pachter2026} proves that
$\popt(Q_d)=\Theta((4/3)^d)$ as $d\to\infty$.
This is an asymptotic statement, not an exact formula for
each dimension.

For $d\ge2$, the standard Keller graph $G_d$ on $\Z_4^d$ is
defined by
\begin{equation}\label{eq:keller-S}
 S_d=
 \{x\in\Z_4^d:
   |\supp(x)|\ge2,\ x_i=2\text{ for some }i\},
 \qquad
 G_d=\Cay(\Z_4^d,S_d).
\end{equation}
It is vertex-transitive, with order $4^d$ and degree
$4^d-3^d-d$; see \cite[Section~6]{Jarnicki2017}.
The coordinate checks in Lemma~\ref{lem:keller-geometry} show that
$G_d$ has diameter two and no dominating edge.
Consequently, \eqref{eq:diameter-cayley} yields
\begin{equation}\label{eq:keller-static}
 \pi(G_d)=4^d,\qquad
 \popt(G_d)=4
 \qquad(d\ge2).
\end{equation}

We study generator reinforcement under the specified regular action:
the aim is to add as few generators as possible in an inverse-closed
set so that the augmented graph has optimal pebbling number at most
three. Definition~\ref{def:rho} formalizes the resulting parameter.
We obtain an inverse-orbit covering formula
(Theorem~\ref{thm:cover}), an algorithm for abelian groups
(Theorem~\ref{thm:abelian}), formulas for central involutions
and cubelike graphs
(Theorem~\ref{thm:involution} and Corollary~\ref{cor:cubelike}),
and exact reinforcement thresholds for hypercubes and standard
Keller graphs (Corollary~\ref{cor:cube} and
Theorem~\ref{thm:keller-reinforcement}).
These thresholds are $2^{d-1}-d$ for $Q_d$ ($d\ge1$), and
$3$ for $G_2$ and $2^{d-1}$ for $G_d$ ($d\ge3$).
Every admissible augmentation of either family remains Class~$0$.
Below the threshold, the optimal pebbling number is at least four
for hypercubes and exactly four for Keller graphs.

\section{\bf Inverse-orbit covering for generator reinforcement}
\label{sec:reinforcement}

\begin{definition}\label{def:rho}
Let $(\Gamma,S)$ satisfy the conventions of
Section~\ref{sec:intro}, and put $T=\{e\}\cup S$.
An \emph{admissible augmentation} is a set
$R\subseteq\Gamma\setminus T$ satisfying $R=R^{-1}$.
Define
$$
 \rhothree(\Gamma,S)
 =
 \min\bigl\{
 |R|:
 R\subseteq\Gamma\setminus T,\ R=R^{-1},\
 \popt\bigl(\Cay(\Gamma,S\cup R)\bigr)\le3
 \bigr\}.
$$
For each $s\in\Gamma\setminus\{e\}$, let
$\rho_{3,s}(\Gamma,S)$ be the minimum size of an admissible
augmentation $R$ for which $\{e,s\}$ is a dominating edge of
$\Cay(\Gamma,S\cup R)$.
\end{definition}

All these minima are finite, since $R=\Gamma\setminus T$
produces the complete graph.
The cost counts individual generators: an inverse pair of
distinct elements costs two, whereas an involution costs one.
An admissible augmentation preserves the specified regular
action, but need not preserve every automorphism of the
original graph.
Thus $\rhothree(\Gamma,S)$ is a parameter of the specified
pair $(\Gamma,S)$. No claim is made that this parameter is independent of the chosen regular action.
Every admissible augmentation adds exactly $N|R|/2$
undirected edges, by the degree-sum formula.

By~\eqref{eq:small}, an augmented graph has optimal pebbling
number at most three if and only if it has a dominating edge.
Left translation takes any dominating edge $\{u,v\}$ to
$\{e,u^{-1}v\}$.
Consequently,
\begin{equation}\label{eq:min-witness}
 \rhothree(\Gamma,S)
 =
 \min_{s\in\Gamma\setminus\{e\}}
 \rho_{3,s}(\Gamma,S).
\end{equation}

Fix $s\ne e$ and put
$$
 F_s=\{s,s^{-1}\}\setminus S,\qquad
 T_s=T\cup\{s,s^{-1}\},\qquad
 M_s^*=\Gamma\setminus(T_s\cup sT_s).
$$
Both $F_s$ and $T_s$ are inverse-closed.
The \emph{inverse orbits} in $\Gamma\setminus T_s$ are the sets
$[x]=\{x,x^{-1}\}$, with weights
$$
 w([x])=|[x]|\in\{1,2\}.
$$
Construct a possibly disconnected auxiliary graph $\mathcal H_s$
whose vertices are these inverse orbits, including, for each
$g\in M_s^*$, the edge
\begin{equation}\label{eq:orbit-edge}
 \bigl\{[g],[s^{-1}g]\bigr\}.
\end{equation}
Both endpoints are well defined: $g\notin T_s$ and
$s^{-1}g\notin T_s$ by the definition of $M_s^*$.
Loops are permitted and must be covered by their incident vertex.
Duplicate edges represent the same constraint and are retained only once.
Write $\tau_w(\mathcal H_s)$ for the minimum total weight of
a vertex cover of $\mathcal H_s$, with this convention about
loops.

\begin{theorem}
\label{thm:cover}
For every pair $(\Gamma,S)$ satisfying the conventions of
Section~\ref{sec:intro} and every $s\ne e$,
$$
 \rho_{3,s}(\Gamma,S)
 =
 |F_s|+\tau_w(\mathcal H_s).
$$
Consequently,
$$
 \rhothree(\Gamma,S)
 =
 \min_{s\in\Gamma\setminus\{e\}}
 \bigl(|F_s|+\tau_w(\mathcal H_s)\bigr).
$$
For each fixed $s$, an optimal weighted vertex cover of
$\mathcal H_s$, together with the forced generators $F_s$,
yields an augmentation attaining $\rho_{3,s}(\Gamma,S)$.
Choosing $s$ to minimize the displayed expression yields an
augmentation attaining $\rhothree(\Gamma,S)$.
\end{theorem}

\begin{proof}
For $\{e,s\}$ to be an edge of the augmented graph, its
generating set must contain $s$ and, by inverse closure,
$s^{-1}$.
Thus every admissible augmentation for this fixed witness
must contain $F_s$.

Such an augmentation has a unique decomposition
$$
 R=F_s\mathbin{\dot\cup}R',
 \qquad
 R'\subseteq\Gamma\setminus T_s,
 \qquad
 R'=(R')^{-1}.
$$
The closed-neighborhood union of $\{e,s\}$ in the augmented
graph is
$$
 (T_s\cup R')\cup s(T_s\cup R')
 =
 T_s\cup sT_s\cup R'\cup sR'.
$$
Therefore only the vertices of $M_s^*=\Gamma\setminus(T_s\cup sT_s)$ remain to be covered
after the forced additions.
For $g\in M_s^*$, membership in the augmented
closed-neighborhood union is equivalent to
$$
 g\in R'
 \quad\text{or}\quad
 s^{-1}g\in R'.
$$
Since $R'$ is a union of inverse orbits, this is exactly the
vertex-cover constraint associated with~\eqref{eq:orbit-edge},
including the case of a loop.

Conversely, if $\mathcal C$ is a vertex cover of
$\mathcal H_s$, then
$$
 R_{\mathcal C}
 =
 F_s\cup\bigcup_{O\in\mathcal C}O
$$
is an admissible augmentation for which $\{e,s\}$ is a
dominating edge.
The inverse orbits are pairwise disjoint and disjoint from
$F_s$, hence
$$
 |R_{\mathcal C}|
 =
 |F_s|+\sum_{O\in\mathcal C}w(O).
$$
The correspondence therefore preserves costs and proves
the fixed-witness formula.
Taking the minimum over $s\ne e$ and
using~\eqref{eq:min-witness} gives the global formula and
the assertions about optimal augmentations.
\end{proof}

The covering formulation determines which generators to add;
it does not prescribe a sequence of pebbling moves.

\begin{theorem}\label{thm:abelian}
Let $s\ne e$ and let $\mathcal H_s$ be the auxiliary graph defined
by~\eqref{eq:orbit-edge}, with duplicate constraints represented
by a single edge. If $s$ is central, then deleting all vertices
carrying loops, together with their incident edges, leaves a
simple graph of maximum degree at most two.
An optimal weighted vertex cover of $\mathcal H_s$,
and hence $\rho_{3,s}(\Gamma,S)$ and an augmentation attaining
it, can be computed in $O(N)$ operations.

For an abelian group, all values $\rho_{3,s}(\Gamma,S)$,
their minimum $\rhothree(\Gamma,S)$, and an augmentation
attaining this minimum can be computed in $O(N^2)$ operations
using $O(N)$ auxiliary storage beyond the input.

These operation counts assume a word-RAM model with
$O(\log N)$-bit words, explicitly indexed group elements,
constant-time multiplication and inversion lookup, and
Boolean membership arrays for subsets.
If only a multiplication table is supplied, an inverse table
can first be constructed in $O(N^2)$ operations using $O(N)$
additional storage. This preprocessing does not change the
total abelian bound.
\end{theorem}

\begin{proof}
Fix a central element $s\ne e$ and let $\mathcal H_s$ be the auxiliary graph defined by~\eqref{eq:orbit-edge}, with duplicate constraints identified. Let $\mathcal L_s$ be the set of vertices carrying a loop. Every vertex cover of $\mathcal H_s$ contains $\mathcal L_s$, because the loop at $O$ can only be covered by selecting $O$. Thus we first select all vertices in $\mathcal L_s$, pay their weights, and remove them together with their incident edges. Let $K_s=\mathcal H_s-\mathcal L_s$.

We claim that \(K_s\) has maximum degree at most two. Fix an inverse orbit $[x]=\{x,x^{-1}\}$. A constraint has the form
\[
\{[g],[s^{-1}g]\},\qquad g\in M_s^*.
\]
If it is incident with \([x]\) through its first endpoint, then \(g=x\) or \(g=x^{-1}\), and its other endpoint is either
\[
[s^{-1}x]\quad\text{or}\quad [s^{-1}x^{-1}]=[sx].
\]
If it is incident with \([x]\) through its second endpoint, then \(s^{-1}g=x\) or \(s^{-1}g=x^{-1}\), so the other endpoint is either
\[
[sx]\quad\text{or}\quad [sx^{-1}]=[s^{-1}x].
\]
Hence all nonloop neighbors of \([x]\) lie in
\[
\{[sx],[s^{-1}x]\}.
\]
Therefore \(K_s\) is a simple graph of maximum degree at most two. Its connected components are isolated vertices, paths, or cycles.

For each path component \(v_1\ldots v_t\) with weights \(w_1,\ldots,w_t\), let \(A_i\) and \(B_i\) denote the minimum cover costs for the first \(i\) vertices when \(v_i\) is selected and unselected, respectively. Then
\[
A_1=w_1,\qquad B_1=0,
\]
and for \(2\le i\le t\),
\[
A_i=w_i+\min(A_{i-1},B_{i-1}),\qquad
B_i=A_{i-1}.
\]
The minimum cost is \(\min(A_t,B_t)\). For a cycle \(v_1\ldots v_t\), we run the same recurrence twice: first with \(A_1=w_1\), \(B_1=+\infty\), and then with \(A_1=+\infty\), \(B_1=0\). In the first run the closing edge is covered by the selected vertex \(v_1\), so the cost is \(\min(A_t,B_t)\). In the second run \(v_1\) is unselected, so the closing edge forces \(v_t\) to be selected and the cost is \(A_t\). Taking the smaller of these two values gives the optimal cycle cost. Storing predecessor choices during the recurrences allows the corresponding cover to be recovered by backtracking. Isolated vertices are not selected.

The preceding computations take time linear in
$|V(\mathcal H_s)|+|E(\mathcal H_s)|=O(N)$.
To construct the auxiliary graphs, build the inverse-orbit partition
of $\Gamma$ once and store each element's orbit identifier and each
orbit's weight. For a fixed $s$, form $T_s$ and $M_s^*$ in $O(N)$
time using Boolean membership arrays, retain the orbits outside
$T_s$, and scan the elements $g\in M_s^*$ to insert the constraints
in~\eqref{eq:orbit-edge}.
Use a separate flag for a loop at each vertex.
The neighbor calculation above shows that every vertex of
$\mathcal H_s$ has at most two distinct nonloop neighbors, so
duplicate constraints can be removed in constant time each.
After the loop vertices are selected and removed, component
traversal, dynamic programming, and backtracking take $O(N)$
time and use $O(N)$ space.

Given an optimal cover \(C_s\) of \(\mathcal H_s\), reconstruct the augmentation by
\[
R_s=F_s\cup\bigcup_{O\in C_s}O.
\]
This set is admissible and attains \(\rho_{3,s}(\Gamma,S)\) by Theorem~\ref{thm:cover}. Its size is
\[
|F_s|+\sum_{O\in C_s}w(O).
\]
When \(\Gamma\) is abelian, every \(s\neq e\) is central, so repeating the fixed-witness computation for all \(N-1\) witnesses takes \(O(N^2)\) time. Reusing the working arrays and retaining one best cover keeps the additional storage at \(O(N)\). If no inverse table is supplied, it can be constructed from the multiplication table in \(O(N^2)\) time and \(O(N)\) additional space, which does not change the overall bound.
\end{proof}

\begin{remark}
Theorem~\ref{thm:abelian} solves the generator-reinforcement
problem with target optimal pebbling number at most three.
It does not provide an algorithm for computing $\pi(G)$ or
$\popt(G)$ for arbitrary abelian Cayley graphs.

For a general nonabelian group,
Theorem~\ref{thm:cover} remains exact.
The linear-time fixed-witness algorithm also applies whenever
the chosen witness is central, but the degree-two argument
does not extend to arbitrary noncentral witnesses.
\end{remark}

\section{\bf Central involutions and cubelike graphs}
\label{sec:involutions}

Throughout this section, $T=\{e\}\cup S$ denotes the
\emph{original} closed generating neighborhood.
The following formula identifies the additional cost imposed
by inverse closure when the prescribed witness is a central
involution.

\begin{theorem}\label{thm:involution}
Let $s\ne e$ be a central involution of $\Gamma$, and put
$$
 M_s=\Gamma\setminus(T\cup sT),\qquad
 r_s=\bigl|\{x\in M_s:x^2=s\}\bigr|.
$$
Then
\begin{equation}\label{eq:involution}
 \rho_{3,s}(\Gamma,S)
 =
 \ind_{\{s\notin S\}}
 +\frac{|M_s|+r_s}{2}.
\end{equation}
\end{theorem}

\begin{proof}
By Theorem~\ref{thm:cover},
$$
 \rho_{3,s}(\Gamma,S)=|F_s|+\tau_w(\mathcal H_s).
$$
Since \(s\) is an involution, \(F_s=\varnothing\) when \(s\in S\), and \(F_s=\{s\}\) when \(s\notin S\). Hence \(|F_s|=\ind_{\{s\notin S\}}\).

Put $T_s=T\cup\{s\}$. In either case,
$$
 T_s\cup sT_s
 =(T\cup\{s\})\cup s(T\cup\{s\})
 =T\cup sT,
$$
because $s^2=e$ and $e\in T$. Thus $M_s^*=M_s$:
the forced addition does not change the set still to be covered.

The complement \(M_s=\Gamma\setminus(T\cup sT)\) is invariant under multiplication by \(s\), since
$$
 s(T\cup sT)=sT\cup T.
$$
 It is also inverse-closed, because \(T=T^{-1}\), \(s=s^{-1}\), and \(s\) is central:
$$
 (T\cup sT)^{-1}=T\cup sT.
$$
If an additional generator $r$ lies outside $M_s$, then both $r$
and $sr$ lie in $T\cup sT$, so it covers no vertex of $M_s$.
Since $M_s$ is inverse-closed, discarding such generators
preserves inverse closure.
After the forced addition, it therefore suffices to choose an
inverse-closed subset of $M_s$ meeting every pair $\{x,sx\}$.

On \(M_s\), the maps \(L_s(x)=sx\) and \(\iota(x)=x^{-1}\) are commuting involutions, and \(L_s\) has no fixed point. Their joint orbits are therefore of size two or four. We analyze each joint orbit \(J\).

If \(|J|=4\), write
$$
 J=\{x,sx,x^{-1},sx^{-1}\}.
$$
The two inverse orbits are
$$
 \{x,x^{-1}\},\qquad \{sx,sx^{-1}\}.
$$
Selecting either one covers both \(L_s\)-pairs and costs two. Since inversion has no fixed point on \(J\), this is minimal.

If \(|J|=2\) and inversion fixes both points, then \(J=\{x,sx\}\), with both elements self-inverse. Selecting either element covers the unique \(L_s\)-pair and costs one.

If \(|J|=2\) and inversion interchanges the two points, then \(J=\{x,sx\}\) and \(x^{-1}=sx\). Inverse closure forces both elements to be selected, so the cost is two. This case is equivalently characterized by \(x^2=s\). Indeed,
$$
 x^{-1}=sx\iff e=xsx=sx^2\iff x^2=s.
$$
If \(x\) is exceptional, then \(sx=x^{-1}\) is exceptional as well, and these two points are distinct. Hence the exceptional points form \(r_s/2\) pairs, one for each joint orbit of this third type.

Summing over all joint orbits, the first two types contribute \(|J|/2\), while each exceptional orbit contributes \(|J|/2+1\). Since the joint orbits partition \(M_s\), the total cover cost is
$$
\frac{|M_s|}{2}+\frac{r_s}{2}
=\frac{|M_s|+r_s}{2}.
$$
Adding the forced cost \(\ind_{\{s\notin S\}}\) gives
$$
\rho_{3,s}(\Gamma,S)
=\ind_{\{s\notin S\}}+\frac{|M_s|+r_s}{2}.
$$
Taking the union of the minimizing choices in all joint orbits produces an inverse-closed set satisfying all constraints, so the bound is attained.

\end{proof}

\begin{corollary}\label{cor:cubelike}
Let $S\subseteq\F_2^d\setminus\{0\}$ generate $\F_2^d$,
where $d\ge1$.
In additive notation, put $T=\{0\}\cup S$ and
$$
 c(s)=|S\cap(s+S)|
 \qquad(s\in\F_2^d\setminus\{0\}),
$$
whether or not $s$ belongs to $S$.
Then, for every $s\ne0$,
$$
 \rho_{3,s}(\F_2^d,S)
 =
 2^{d-1}-|S|+\frac{c(s)}2.
$$
Consequently,
\begin{equation}\label{eq:cubelike}
 \rhothree(\F_2^d,S)
 =
 2^{d-1}-|S|
 +\frac12\min_{s\ne0}c(s).
\end{equation}
For a fixed $s$, an augmentation attaining
$\rho_{3,s}(\F_2^d,S)$ is obtained by adding $s$ if it is
absent and selecting one element from each pair
$\{x,x+s\}$ in
$$
 M_s=\F_2^d\setminus\bigl(T\cup(s+T)\bigr).
$$
Choosing $s$ to minimize $c(s)$ gives an augmentation attaining
$\rhothree(\F_2^d,S)$.
\end{corollary}

\begin{proof}
Fix \(s\neq 0\). Since \(\mathbb F_2^d\) is abelian and \(2s=0\), \(s\) is a central involution. In additive notation, Theorem~\ref{thm:involution} gives
$$
\rho_{3,s}(\mathbb F_2^d,S)
 =\ind_{\{s\notin S\}}+\frac{|M_s|+r_s}{2}.
$$
The equation \(x^2=s\) becomes \(2x=s\), which has no solution because \(2x=0\). Hence \(r_s=0\).

We compute \(|M_s|\). Write
$$
T=\{0\}\cup S,\qquad s+T=\{s\}\cup(s+S).
$$
If \(s\notin S\), then
$$
T\cap(s+T)=S\cap(s+S),
$$
so its size is \(c(s)\). If \(s\in S\), the intersection additionally contains \(s\) and \(0\), and these two elements are not in \(S\cap(s+S)\). Hence
$$
|T\cap(s+T)|=c(s)+2.
$$
Thus
$$
|T\cap(s+T)|=c(s)+2\ind_{\{s\in S\}}.
$$
Since \(|T|=|s+T|=|S|+1\),
$$
|M_s|
=2^d-\bigl(2(|S|+1)-|T\cap(s+T)|\bigr)
=2^d-2(|S|+1)+c(s)+2\ind_{\{s\in S\}}.
$$
Substituting into~\eqref{eq:involution} gives
$$
\begin{aligned}
\rho_{3,s}
&=\ind_{\{s\notin S\}}
+\frac{2^d-2(|S|+1)+c(s)+2\ind_{\{s\in S\}}}{2}\\
&=2^{d-1}-|S|+\frac{c(s)}2.
\end{aligned}
$$

The map \(x\mapsto x+s\) preserves \(S\cap(s+S)\) and has no fixed point, so \(c(s)\) is even. Likewise, \(x\mapsto x+s\) partitions \(M_s\) into pairs \(\{x,x+s\}\). Choosing one element from each pair covers all constraints. Inverse closure is automatic in \(\mathbb F_2^d\). Together with \(s\) when \(s\notin S\), this attains the fixed-witness minimum. Taking the minimum over \(s\neq0\) gives~\eqref{eq:cubelike}.

\end{proof}

\begin{corollary}\label{cor:cube}
For the standard coordinate generating set of $Q_d$,
where $d\ge1$,
$$
 \rhothree(\F_2^d,\{e_1,\ldots,e_d\})
 =
 2^{d-1}-d.
$$
For $d\ge3$, every admissible augmentation with fewer than
$2^{d-1}-d$ additional generators has optimal pebbling
number at least four.
For every $d\ge2$, there is an augmentation with exactly
$2^{d-1}-d$ additional generators and optimal pebbling
number exactly three.
Every admissible augmentation, regardless of its size,
has ordinary pebbling number $2^d$.
\end{corollary}

\begin{proof}
For \(S=\{e_1,\ldots,e_d\}\), the formula for \(c(s)\) is immediate. If $\wt(s)=2$, write $s=e_i+e_j$ with $i\ne j$. Its two ordered representations as a sum of elements of $S$ are $e_i+e_j$ and $e_j+e_i$, so $c(s)=2$. For nonzero $s$ with $\wt(s)\ne2$, no such representation exists, so $c(s)=0$. Every coordinate vector has weight one, so \(\min_{s\neq0}c(s)=0\), and Corollary~\ref{cor:cubelike} gives
$$
\rho_3(\mathbb F_2^d,\{e_1,\ldots,e_d\})
=2^{d-1}-d.
$$

Choose a coordinate vector as the witness. Corollary~\ref{cor:cubelike} constructs an admissible augmentation adding \(2^{d-1}-d\) generators. Since \(Q_d\) has degree \(d\), the augmented graph has degree
$$
d+(2^{d-1}-d)=2^{d-1}.
$$
For \(d\ge2\), this is strictly smaller than \(2^d-1\), so the graph is noncomplete and has no universal vertex. The construction gives a dominating edge, so~\eqref{eq:small} implies that its optimal pebbling number is exactly three.

An augmentation with fewer than \(2^{d-1}-d\) generators cannot have \(\pi_{\mathrm{opt}}\le3\), by the definition of $\rhothree$. Its optimal pebbling number is therefore at least four.

For the ordinary pebbling number, every \(n\)-vertex graph is a subgraph of \(K_n\), and \(\pi(K_n)=n\). Since adding edges cannot increase the ordinary pebbling number,
$$
\pi(G)\ge n
$$
for every \(n\)-vertex graph \(G\). Hence every admissible augmentation \(H\) satisfies
$$
2^d=\pi(Q_d)\ge\pi(H)\ge 2^d,
$$
and therefore \(\pi(H)=2^d\).

Finally, each new generator contributes \(2^{d-1}\) undirected edges. At the threshold, the number of new generators is \(2^{d-1}-d\), so the number of new edges is
$$
2^{d-1}(2^{d-1}-d).
$$
\end{proof}

In dimensions one and two the threshold is zero, reflecting $\popt(K_2)=2$ and $\popt(C_4)=3$, respectively.
For $d\ge3$, the subthreshold conclusion is only the lower
bound $\popt\ge4$; equality is not asserted.

\begin{example}
\label{ex:separation}
Represent the vectors of $\F_2^4$ by binary integers, with
addition given by bitwise exclusive-or, and take
$$
 S_A=\{1,2,3,4,8,12\},
 \qquad
 S_B=\{1,2,4,7,8,15\}.
$$
Let
$$
 G_A=\Cay(\F_2^4,S_A),
 \qquad
 G_B=\Cay(\F_2^4,S_B).
$$
For $X\in\{A,B\}$, write $c_X(s)=|S_X\cap(s+S_X)|$.
Both graphs have order $16$, degree $6$, diameter $2$,
ordinary pebbling number $16$, and optimal pebbling
number $4$, but
$$
 \rhothree(\F_2^4,S_A)=3,
 \qquad
 \rhothree(\F_2^4,S_B)=2.
$$

Indeed, put
$$
 U=\{0,1,2,3\},
 \qquad
 V=\{0,4,8,12\}.
$$
These are complementary two-dimensional subspaces of
$\F_2^4$, and
$$
 S_A=(U\setminus\{0\})\cup(V\setminus\{0\}).
$$
Thus $G_A\cong K_4\square K_4$.
Every nonzero vector has exactly two ordered representations
as a sum of two elements of $S_A$, so
$$
 c_A(s)=2\qquad(s\ne0).
$$

The set $S_B$ contains the coordinate basis and hence
generates $\F_2^4$.
Its nine nonzero elements not belonging to $S_B$ are
$$
 3,5,6,9,10,11,12,13,14,
$$
and each is a sum of two elements of $S_B$.
Therefore $G_B$ also has diameter two.
Moreover,
$$
 S_B\cap(1+S_B)=\varnothing,
$$
so $\min_{s\ne0}c_B(s)=0$.
Formula~\eqref{eq:cubelike} now gives
$$
 \rhothree(\F_2^4,S_A)=8-6+1=3,
 \qquad
 \rhothree(\F_2^4,S_B)=8-6=2.
$$
The common ordinary and optimal pebbling numbers follow from~\eqref{eq:diameter-cayley} and $16>2\cdot6$.

To construct minimum-cost augmentations for these two graphs, take the witness $s=1$.
The missing pairs for $S_A$ are
$$
 \{6,7\},\qquad \{10,11\},\qquad \{14,15\},
$$
whereas those for $S_B$ are
$$
 \{10,11\},\qquad \{12,13\}.
$$
Thus one may choose
$$
 R_A=\{6,10,14\},
 \qquad
 R_B=\{10,12\}.
$$
In each augmented graph, $\{0,1\}$ is a dominating edge.

This example shows that the order, degree, diameter, and
the two static pebbling numbers, even taken together,
do not determine the generator-reinforcement parameter.
\end{example}

\section{\bf Exact reinforcement of standard Keller graphs}
\label{sec:keller}

Let $S_d$ and $G_d$ be as in~\eqref{eq:keller-S}, and put
$T=\{0\}\cup S_d$.
All coordinate operations are performed modulo four, and
all reinforcement costs refer to the specified $\Z_4^d$ action.
The following coordinate checks establish the hypotheses needed
to apply~\eqref{eq:diameter-cayley}.

\begin{lemma}\label{lem:keller-geometry}
For every $d\ge2$, the graph $G_d$ has diameter two and
no dominating edge.
\end{lemma}

\begin{proof}
If \(x\in T\), then \(x=x+0\in T+T\). Hence assume \(x\notin T\). If $x$ has no coordinate equal to $2$,
choose distinct coordinates $i,j$ and put
$$
 a=2e_i+(x_j-2)e_j,\qquad b=x-a.
$$
The $i$th and $j$th coordinates of $a$ are $2$ and a
nonzero element, whereas those of $b$ are a nonzero
element and $2$. Hence $a,b\in S_d$ and $x=a+b$.
Otherwise, $x=2e_i$ for some $i$. Choosing $j\ne i$ gives
$$
 x=(e_i+2e_j)+(e_i+2e_j),
$$
with both summands in $S_d$.
Thus $T+T=\Z_4^d$. The vertices \(0\) and \(2e_i\) are nonadjacent because their difference is \(2e_i\notin S_d\). Hence \(G_d\) is not complete. Together with \(T+T=\mathbb Z_4^d\), this gives diameter exactly two.

For any $s\in S_d$, define $w_i=1$ when $s_i=2$ and
$w_i=0$ otherwise. Then $w\ne0,s$, and neither $w$ nor
$w-s$ has a coordinate equal to $2$.
Therefore $w\notin T\cup(s+T)$.
For an arbitrary edge $\{u,v\}$, put $s=v-u\in S_d$.
Translation by $-u$ maps this edge to $\{0,s\}$.
The vertex $w$ constructed above lies outside
$T\cup(s+T)$, so $u+w$ lies outside
$(u+T)\cup(v+T)=N[u]\cup N[v]$.
Thus no edge is dominating.
\end{proof}

\begin{theorem}
\label{thm:keller-reinforcement}
For every $d\ge2$,
\begin{equation}\label{eq:keller-threshold}
 \rhothree(\Z_4^d,S_d)=
 \begin{cases}
  3,&d=2,\\
  2^{d-1},&d\ge3.
 \end{cases}
\end{equation}
Every admissible augmentation with fewer additional generators
than this threshold has optimal pebbling number exactly four.
There is an augmentation attaining the threshold whose
optimal pebbling number is exactly three.
Every admissible augmentation has ordinary pebbling number
$4^d$.
\end{theorem}

\begin{proof}
We first establish a lower bound for every nonzero $s$, whether or not $s\in S_d$.
Fix $s\in\Z_4^d\setminus\{0\}$, and let $R$ be an admissible
augmentation for which $\{0,s\}$ is a dominating edge.
Writing
$$
 M_s=\Z_4^d\setminus\bigl(T\cup(s+T)\bigr),
$$
we necessarily have
$$
 M_s\subseteq R\cup(s+R).
$$

Define the coordinate box
$$
 B_s=
 \prod_{i=1}^d
 \bigl(\Z_4\setminus\{2,s_i+2\}\bigr).
$$
If $z(s)$ is the number of zero coordinates of $s$, then
$$
 |B_s|=3^{z(s)}2^{d-z(s)}\ge2^d.
$$
For every $x\in B_s$, neither $x$ nor $x-s$ has a coordinate
equal to $2$.
Consequently,
$$
 B_s\setminus\{0,s\}\subseteq M_s.
$$

Suppose first that $s$ has a coordinate equal to $2$.
In that coordinate, the box permits only $1$ and $3$,
so neither $0$ nor $s$ belongs to $B_s$.
It follows that
$$
 B_s\subseteq M_s\subseteq R\cup(s+R),
$$
and hence
$$
 2|R|\ge|B_s|\ge2^d.
$$

Suppose instead that $s$ has no coordinate equal to $2$.
Since $s\ne0$, it has an odd coordinate and therefore has
order four.
Neither $s$ nor $-s$ belongs to $S_d$, so the prescribed
edge and inverse closure force
$$
 \{s,-s\}\subseteq R.
$$
Both $0$ and $s$ belong to $B_s$, whereas $-s$ and $2s$
do not.
Indeed, at an odd coordinate,
$-s_i=s_i+2$ and $2s_i=2$.
Thus
$$
 \bigl(\{s,-s\}\cup(s+\{s,-s\})\bigr)\cap B_s
 =
 \{0,s\}.
$$
The remaining vertices of $B_s$ must be covered by the
other added generators and their translates, giving
$$
 |B_s|-2\le2(|R|-2).
$$
For every other generator $r$, its contribution to $B_s$ is contained in $\{r,s+r\}$.
Therefore
$$
 |R|\ge\frac{|B_s|}{2}+1\ge2^{d-1}+1.
$$
Combining the two cases shows that every successful
augmentation has at least $2^{d-1}$ additional generators.

For $d\ge3$, this lower bound is attained.
Choose the existing generator
$$
 s=(\underbrace{2,\ldots,2}_{d-1},1)
$$
and the augmentation
$$
 R=\{1,3\}^{d-1}\times\{0\}.
$$
The set $R$ is inverse-closed, disjoint from $T$, and has
size $2^{d-1}$.
We claim that
\begin{equation}\label{eq:keller-missing}
 \Z_4^d\setminus\bigl(T\cup(s+T)\bigr)
 =
 \{1,3\}^{d-1}\times\{0,1\}
 =
 R\cup(s+R).
\end{equation}

To verify the first equality, recall that a nonzero vector
outside $T$ either has no coordinate equal to $2$ or is
one of the exceptional vectors $2e_i$.
Every $2e_i$ belongs to $s+S_d$.
For $i<d$, the vector $2e_i-s$ has a coordinate equal to
$2$ among the other first $d-1$ coordinates, since
$d-2\ge1$, and also has a nonzero last coordinate.
For $i=d$, it likewise has at least two nonzero coordinates
and a coordinate equal to $2$.
Thus $2e_i-s\in S_d$ in every case.
Similarly,
$$
 s+2e_i\in S_d\qquad(1\le i\le d).
$$
Hence a vertex outside both closed neighborhoods cannot
be exceptional relative to either endpoint.

It follows that, for every $x\in M_s$, neither $x$ nor $x-s$
has a coordinate equal to $2$.
For the first $d-1$ coordinates, this requires
$x_i\in\{1,3\}$; for the last coordinate, it requires
$x_d\in\{0,1\}$.
Conversely, every vector satisfying these restrictions
lies outside both closed neighborhoods, and the resulting
box contains neither endpoint.
This proves the first equality in~\eqref{eq:keller-missing}.
The second follows because translation by $s$ interchanges
$1$ and $3$ in the first $d-1$ coordinates and changes the
last coordinate from $0$ to $1$.
Thus $\{0,s\}$ is a dominating edge after adding $R$.

It remains to determine the threshold in dimension two. Witnesses $s$ with no coordinate equal to $2$ already require at least three additional generators.
For the other differences, coordinate permutations and
independent sign changes preserve $S_2$ and the reinforcement
costs.
There are therefore only three representatives to consider:
$$
 (2,0),\qquad (2,2),\qquad (2,1).
$$

For the two involutory witnesses, the sets $M_s$ for the original graph are
$$
 M_{(2,0)}
 =
 \{1,3\}\times\{0,1,3\},
$$
and
$$
 M_{(2,2)}
 =
 \{1,3\}^2\cup\{(0,2),(2,0)\}.
$$
Both have six elements.
The numbers of solutions of $2x=s$ in these sets are two
and four, respectively.
The witness $(2,0)$ is absent from $S_2$, whereas $(2,2)$
belongs to $S_2$.
Thus~\eqref{eq:involution}, in additive notation, gives
$$
 \rho_{3,(2,0)}(\Z_4^2,S_2)
 =
 1+\frac{6+2}{2}=5,
$$
and
$$
 \rho_{3,(2,2)}(\Z_4^2,S_2)
 =
 \frac{6+4}{2}=5.
$$

For $s=(2,1)$, the generating set is
$$
 S_2=
 \{(2,1),(2,2),(2,3),(1,2),(3,2)\},
$$
and direct calculation gives
$$
 M_s=
 \{(0,1),(1,0),(1,1),(2,0),(3,0),(3,1)\}.
$$
Since $(0,1)\notin T\cup(s+T)$, covering it requires
either $(0,1)$ or
$$
 (0,1)-s=(2,0)
$$
to belong to $R$.
Thus at least one added generator has even first coordinate.
Covering the missing vertex $(1,0)$ requires either
$(1,0)$ or
$$
 (1,0)-s=(3,3)
$$
to belong to $R$.
Both choices have order four and odd first coordinate.
Inverse closure therefore requires two distinct added
generators with odd first coordinate, disjoint from the
preceding even-coordinate choice.
Hence $|R|\ge3$.

This bound is attained by
$$
 R=\{(1,0),(2,0),(3,0)\}.
$$
Indeed, $R$ is admissible and, for $s=(2,1)$,
$$
 M_s=R\cup(s+R).
$$
Thus $\{0,s\}$ becomes a dominating edge.
This completes the proof of~\eqref{eq:keller-threshold}.

The exhibited augmentations are noncomplete. The number of group elements available as new generators is
$$
 |\Z_4^d\setminus T|=3^d+d-1,
$$
which is larger than the threshold for every $d\ge2$.
The augmented graphs are therefore regular and noncomplete,
so they have no universal vertex.
Together with their dominating edges,
\eqref{eq:small} gives optimal pebbling number exactly three.

For augmentations below the threshold, the optimal pebbling number exceeds three by definition of $\rhothree$.
It is at most four because adding edges cannot increase the
optimal pebbling number and $\popt(G_d)=4$.
Thus every subthreshold augmentation has optimal pebbling
number exactly four.

Finally, every augmented graph contains $G_d$ as a spanning
subgraph.
By~\eqref{eq:keller-static}, $\pi(G_d)=4^d$.
If $H$ is an admissible augmentation, monotonicity under edge
addition and the lower bound by the number of vertices give
$$
 4^d\le\pi(H)\le\pi(G_d)=4^d.
$$
Thus $\pi(H)=4^d$ for every admissible augmentation.
\end{proof}

\begin{corollary}\label{cor:antipodal}
For $d\ge3$, prescribing the antipodal difference
$s=(2,\ldots,2)$ gives
$$
 \rho_{3,s}(\Z_4^d,S_d)=2^d,
$$
twice the unrestricted threshold
in~\eqref{eq:keller-threshold}.
\end{corollary}

\begin{proof}
Let
$s=(2,\ldots,2).$
Since \(d\ge3\), \(s\in S_d\), and since \(\mathbb Z_4^d\) is abelian and \(2s=0\), \(s\) is a central involution. Thus Theorem~\ref{thm:involution} applies.

A vertex \(x\notin T\) that has a coordinate equal to 2 must have support size one; otherwise \(x\in S_d\subseteq T\). Hence the only possible exceptional vectors outside \(T\) are \(2e_i\). For every \(i\),
$$
2e_i-s=s+2e_i.
$$
This vector has zero in coordinate \(i\) and \(2\) in all other coordinates. Since \(d\ge3\), it belongs to \(S_d\). Therefore \(2e_i\in s+T\). Similarly, \(s+2e_i\in S_d\), so a vector \(x\) with \(x-s=2e_i\) lies in \(T\).

It follows that if \(x\in M_s\), then neither \(x\) nor \(x-s\) has a coordinate equal to 2. For each coordinate \(i\), this gives
$$
x_i\neq2,\qquad x_i-2\neq2,
$$
hence \(x_i\in\{1,3\}\). Conversely, every \(x\in\{1,3\}^d\) lies outside both closed neighborhoods, so
$
M_s=\{1,3\}^d.
$

For every \(x\in M_s\), each coordinate is 1 or 3, so \(2x_i=2\) and hence \(2x=s\). Therefore
$$
r_s=|M_s|=2^d.
$$
Since \(s\in S_d\), Theorem~\ref{thm:involution} gives
$$
\rho_{3,s}
=0+\frac{2^d+2^d}{2}
=2^d.
$$
\end{proof}

For an augmentation attaining the unrestricted Keller threshold, the number of additional undirected edges is
$$
 \frac{4^d}{2}\,\rhothree(\Z_4^d,S_d)
 =
 \begin{cases}
  24,& \text{if}\; d=2,\\
  2^{3d-2},& \text{if}\; d\ge 3.
 \end{cases}
$$
These are edge costs for augmentations preserving the
specified regular action, rather than reinforcement numbers
allowing arbitrary edge additions.
The square-root correction in
Corollary~\ref{cor:antipodal} also shows why selecting half
of a missing set need not produce an admissible
inverse-closed augmentation.

\section{\bf Conclusion}\label{sec:conclusion}

The results of Sections~\ref{sec:reinforcement}--\ref{sec:keller}
express generator reinforcement under a specified regular action
as a weighted vertex-cover problem on inverse orbits.
For central witnesses, the covering graphs reduce to isolated
vertices, paths, and cycles after the forced loop vertices
are removed.
This yields an $O(N^2)$ algorithm for finite abelian groups
under the computational assumptions of
Theorem~\ref{thm:abelian}.
The central-involution formula identifies the inverse-closure
correction as half the number of square roots of the witness
in the uncovered set. The cubelike
formula gives the exact reinforcement threshold for hypercubes.
Explicit lower bounds and constructions also determine the
threshold for standard Keller graphs.
Example~\ref{ex:separation} shows that the order, degree,
diameter, and the two static pebbling numbers together do
not determine the reinforcement parameter.

These results concern the specified pair $(\Gamma,S)$, with
cost measured by the number of additional generators.
Inverse-closed generator additions preserve the given regular
action, but need not preserve the full automorphism group of
the original graph.
The corresponding number of added edges is $N|R|/2$.
Optimizing over these augmentations is a different problem
from allowing arbitrary individual edge additions.
The abelian algorithm computes the constrained reinforcement
parameter $\rhothree(\Gamma,S)$; it does not compute the
ordinary or optimal pebbling number of an arbitrary abelian
Cayley graph.

\section{\bf Declarations}

\vskip 2mm

\noindent{\bf Data availability} Not applicable.
\vskip 2mm
\noindent{\bf Conflict of interest}  The authors declare that there is no conflict of interest.
\vskip 2mm
\noindent{\bf Funding} This research was supported by NSFC (Nos. 12231009, 12071209).
\vskip 2mm

\noindent{\bf Statement of AI use}
The authors utilized ChatGPT for language editing. All mathematical content and computations were independently verified by the authors.

\end{document}